\documentclass[10pt,reqno]{amsart}
\usepackage{amsmath}
\usepackage{amsfonts}
\usepackage{amssymb}
\usepackage{amsthm}
\usepackage{amscd}

\usepackage{a4wide}

\newcommand{\Section}[1]{\section{#1} \setcounter{equation}{0}}

\newtheorem{thm}{Theorem}
\newtheorem{lem}{Lemma}[section]

\theoremstyle{remark}

\theoremstyle{definition}

\begin{document}
\large
\title{Representation of Integers by a Family of Cubic Forms in Seven Variables II}
\author{Manoj Verma}
\address{MANOJ VERMA: Institute of Mathematical Sciences,
CIT Campus, Taramani, Chennai 600113, India}
\email{mverma@imsc.res.in}

\begin{abstract} In an earlier paper \cite{c7i} we derived asymptotic formulas for the number of representations of zero and of large positive integers by the cubic forms in seven variables which can be written as $L_1(x_1,x_2,x_3) Q_1(x_1,x_2,x_3)+ L_2(x_4,x_5,x_6) Q_2(x_4,x_5,x_6) + a_7 x_7^3$ where $L_1$ and $L_2$ are linear forms, $Q_1$ and $Q_2$ are quadratic forms and $a_7$ is a non-zero integer and for which certain quantities related to $L_1Q_1$ and $L_2Q_2$ were non-zero. In this paper, we consider the case when one or both of these quantities is zero but $L_1Q_1$ and $L_2Q_2$ are still nondegenerate cubic forms in three variables.
\end{abstract}
\maketitle
\Section {Introduction}
\label{intro}
\noindent In an earlier paper \cite{c7i} we derived asymptotic formulas for the number of representations of zero inside a box $|x_i|\leq P$ and the number of
representations of a large positive integer $N$ with
each $x_i \in I$ where $I$ $=\{1,2,\ldots,P\}$ or $\{0,1,2,\ldots,P\}$ or
$\{-P,\ldots,-1,0,1,\ldots,P\}$ with $P=[N^{1/3}]$ by the cubic forms which can be written as
\begin {equation} \label{e:one1}
f({\bf x}) =
L_1
Q_1
+L_2
Q_2
+a_7x_7^3
\end {equation}
where, in case the arguments are not mentioned explicitly, as here,
$$L_1=L_1(x_1,x_2,x_3)=a_1x_1+a_2x_2+a_3x_3,$$
$$L_2=L_2(x_4,x_5,x_6)=a_4x_4+a_5x_5+a_6x_6$$ are linear
forms in three variables,
$$Q_1=Q_1(x_1,x_2,x_3)=A_1
x_1^2+A_2x_2^2+A_3x_3^2+B_1x_2x_3+B_2x_3x_1+B_3x_1x_2,$$
$$Q_2=Q_2(x_4,x_5,x_6)=A_4x_4^2+A_5x_5^2+A_6x_6^2+B_4x_5x_6+B_5x_6x_4+B_6x_4x_5$$ are
quadratic forms in three variables and $a_7$ is a non-zero integer and for which certain quantities related to $L_1Q_1$ and $L_2Q_2$, denoted by $\Delta_1$ and $\Delta_2$ in \cite{c7i} and defined by
\begin{equation}\label{e:one2}
\Delta_1={\bf a}_1{\bf  M}_1{\bf  a}_1^T, \mbox{ } \Delta_2={\bf a}_2{\bf  M}_2{\bf  a}_2^T
\end{equation} were non-zero; here ${\bf a}_1=(a_1, a_2, a_3)$, ${\bf a}_2=(a_4, a_5, a_6)$,
${\bf M}_1$ is the
matrix of the adjoint of the quadratic form $2Q_1$ and ${\bf M}_2$ is the
matrix of the adjoint of the quadratic form $2Q_2$:
$${\bf M}_1 =\left[ \begin{array}{ccc}
B_1^2-4A_2A_3 & 2A_3B_3-B_1B_2 & 2A_2B_2-B_1B_3\\
2A_3B_3-B_1B_2 & B_2^2-4A_1A_3 & 2A_1B_1-B_2B_3\\
2A_2B_2-B_1B_3 & 2A_1B_1-B_2B_3 & B_3^2-4A_1A_2
\end{array}\right],$$
and $M_2$ is given similarly. In this paper we consider the case $\Delta_1=0$ ($\Delta_2$ might or might not be zero). Since the only places where the assumption that $\Delta_1\Delta_2 \neq 0$ was used in \cite{c7i} were the treatment of the minor arcs and the singular series
, we only derive the corresponding estimates for them in case $\Delta_1$ or $\Delta_2$ is zero that suffice to prove the asymptotic formula and omit other details which are the same as in case $\Delta_1\Delta_2 \neq 0$
.\\
\indent As in \cite{c7i},
when considering representation of large positive integers with $x_i$ restricted
to be positive or non-negative, we will assume that $f(x_1,\ldots,x_7)>0$ for a positive proportion of 7-tuples $(x_1,\ldots,x_7)$ with $1 \leq x_i \leq P$. By homogeneity and continuity of $f$
as a function of 7 real variables, this is equivalent to assuming that $f(x_1,\ldots,x_7)>0$
for some $(x_1,\ldots,x_7)\in {\mathbb R}_{\ge 0}^7$. No such condition would be needed if we allow the $x_i$ to take non-positive values as well and require only that $|x_i|\leq
N^{1/3}$. Congruence conditions are still the same and are described in Theorem 2.\\
\indent As usual, $e(\alpha)=e^{2\pi i \alpha}$,
$\varepsilon$ can take any positive real value in
any statement in which it appears and the symbols $\ll$ and $O$ have their usual meanings
with the implicit constants depending on the cubic form $f$ and $\varepsilon$;
any other dependence will be mentioned explicitly through subscripts. By the content of a polynomial (in any number of variables) we shall mean the greatest common divisor of all its coefficients.\\
\indent Since the $h$-invariant of the cubic form $f$ in (\ref{e:one1}) is 3, the hypersurface $f=0$ has one or more four-dimensional rational linear spaces contained in it
. We now describe these spaces and the conditions under which they would arise in terms of $\Delta_2, A'', B'', C'', F'', G''$ (defined in (\ref{e:one2}) above and (\ref{e:two22})-(\ref{e:two27}) below).  (Note that $Q_1$ does not factorize)
.\\
$(1)$ in all cases, the space $\{L_1=L_2=x_7=0\}$;\\
$(2,3)$ if $Q_2$ factorizes, say $Q_2=L_2'L_2''$ (this happens precisely when $\Delta_2$ is a positive square and $D''=0$), then $\{L_1=L_2'=x_7=0\}$ and $\{L_1=L_2''=x_7=0\}$;\\
$(2',3')(i)$ if $\Delta_2$ is a positive square, say $\Delta_2=d^2$, $A''\neq 0$ and $D''/4A''a_4^4\Delta_2a_7$ is the cube of a non-zero rational number, say $D''/4A''a_4^4\Delta_2a_7=d_1^3/d_2^3$, then $\{L_1=a_4dx_5'+x_6'=d_1L_2+d_2x_7=0\}$ and $\{L_1=a_4dx_5'-x_6'=d_1L_2+d_2x_7=0\}$ with $x_5', x_6'$ as in (\ref{e:two23}) (see also (\ref{e:two24}))\\
$(2',3')(ii)$ if $\Delta_2$ is a positive square, say $\Delta_2=d^2$,  $A''=0, C''\neq 0$ and $D''/4C''a_4^4\Delta_2a_7$ is the cube of a non-zero rational number, say $D''/4C''a_4^4\Delta_2a_7=d_1^3/d_2^3$, then $\{L_1=a_4dx_6'+x_5'=d_1L_2+d_2x_7=0\}$ and $\{a_4dx_6'-x_5'=d_1L_1+d_2x_7=0\}$ with  $x_5', x_6'$ as in (\ref{e:two25}) (see also (\ref{e:two26}))\\
$(2',3')(iii)$ if $\Delta_2$ is a positive square, say $\Delta_2=d^2$,  $A''=C''=0, B''\neq 0$ and $D''/B''a_4^2a_7$ is the cube of a non-zero rational number, say $D''/B''a_4^2a_7=d_1^3/d_2^3$, then  $\{L_1=x_5'=d_1L_2+d_2x_7=0\}$ and $\{L_1=x_6'=d_1L_2+d_2x_7=0\}$ with $x_5', x_6'$ as in (\ref{e:two27}) (see also (\ref{e:two28})).\\
\indent The integral points in any of these four-dimensional linear spaces contained in the variety $f=0$ form a four-dimensional lattice and the number of the lattice points in such a space that are contained inside a box $x_i\in I$ is $\sigma P^4+O(P^3)$ for some $\sigma \geq 0$ depending on the space and also on $I$ (we have $\sigma>0$ if $I=\{-P,\ldots,-1,0,1,\ldots,P\}$).\\
\indent Since $Q_1$ does not factorize, it is easily seen by diagonalizing $Q_1$ 
that
\begin{equation*}
\mbox{Card}\{(x_1,x_2,x_3)\in I^3: Q_1(x_1,x_2,x_3)=0\} \ll P^{1+\varepsilon}
\end{equation*}
and a similar statement holds for $Q_2$ in case $Q_2$ does not factorize.\\
\indent Considering all the cases, we see that there exist non-negative constants $\delta_0, \delta_1, \delta_2, \delta_3, \delta_4$ (depending on $f$ and $I$) such that the number of the lattice points on the hypersurface $f=0$ that are contained inside a box $x_i\in I$ and also in the union of all the linear spaces mentioned above is $\delta_0 P^4+O(P^{3})$, the number of those from the cases $(1,2,3)$ is $\delta_1 P^4+O(P^{3})$, the number of those from the cases $(2',3')$ is $\delta_2 P^4+O(P^{3})$, $\delta_2=0$ if the cases $(2',3')$ do not arise,
\begin{equation}\label{e:one3}
N_1:=N_1(I)=\mbox{Card}\{(x_1,x_2,x_3)\in I^3:L_1
Q_1
=0\}=\delta_3 P^2+O(P^{1+\varepsilon}),
\end{equation}
\begin{equation}\label{e:one4}
N_2:=N_2(I)=\mbox{Card}\{(x_4,x_5,x_6)\in I^3:L_2
Q_2
=0\}=\delta_4 P^2+O(P^{1+\varepsilon}),
\end{equation} and
\begin{equation} \label{e:one0}
N_1N_2=\mbox{Card}\{(x_1,\ldots,x_6)\in I^6:L_1Q_1=L_2Q_2=0\}=\delta_1 P^4+O(P^{3+\varepsilon}).
\end{equation}
We have $\delta_1=\delta_3\delta_4$  and $\delta_0=\delta_1+\delta_2$. Also, $\delta_1>0$ if $I=\{-P,\ldots,-1,0,1,\ldots,P\}$.
We are now ready to state our main results.\\
\begin{thm} \label{t:one1}
The number of representations  $R(N)$ of a large positive integer $N$ by the cubic form $f$ in
(\ref{e:one1}) with $x_i \in I$ satisfies
\begin{equation*}
R(N)=\delta_1 N^{4/3} \chi(N)+N^{4/3}\mathfrak{S}(N)J_i+O(N^{{4/3-1/48+\varepsilon}})
\end{equation*}
where $\delta_1\geq 0$ is the constant in (\ref{e:one0}),
\begin{equation}\label{e:one5}
\chi(N)=
\left\{
\begin{array}{cl} 1 & \mbox{if } N=a_7x^3 \mbox{ for some } x\in I\\
0 & \mbox{otherwise}
\end{array}\right.
\end{equation}
is the indicator function of the set $T=\{a_7x^3:x \in I\}$,
$${\mathfrak{S}}(N)=\sum_{q=1}^{\infty}\sum_{\substack{a=1\\(a,q)=1}}
^{q}
q^{-7}S(q,a)e(-aN/q) \mbox{ with }S(q,a)=\sum_{{\bf z} \bmod q} e\left( \frac{a}{q}f({\bf z})\right),$$
$$i=\left\{
\begin{array}{ll} 1 & \mbox{if } I=\{1,2,\ldots,P\} \mbox{ or
}\{0,1,2,\ldots,P\}\\
2 & \mbox{if } I=\{-P,\ldots,-1,0,1,\ldots,P\}
\end{array}\right.,$$
$$J_1=\int_{-\infty}^{\infty}\left( \int_{[0,1]^7}e(\gamma f({\bf \xi}))\,d{\bf
\xi} \right) e(-\gamma)d\gamma$$
and
$$J_2=\int_{-\infty}^{\infty}\left( \int_{[-1,1]^7}e(\gamma f({\bf \xi}))\,d{\bf
\xi} \right) e(-\gamma)d\gamma.$$
Here, $J_2>0$; $J_1>0$ if
$f(x_1,\ldots,x_7)>0$ for some $(x_1,\ldots,x_7)\in {\mathbb R}_{\ge 0}^7$,
$J_1=0$ otherwise;
$1 \ll \mathfrak{S}(N) \ll 1$ if the equation $f({\bf x})=N$ has $p$-adic
solutions for all primes $p$, ${\mathfrak{S}}(N)=0$ otherwise.
\end{thm}
\begin{thm} \label{t:one2}
If $f({\bf x})=c(c_1L_1'Q_1'+c_2L_2'Q_2'+c_3x_7^3)$
where $c$ is the content of $f$ while $L_1'Q_1'$, $L_2'Q_2'$ have content $1$, then the equation
$f({\bf x})=N$ has $p$-adic solutions for all
primes $p$ if the congruence
\begin{equation*}
f({\bf x})\equiv N \left( \mbox{ mod }c\prod_{p|3c_1c_2c_3}p^{\gamma(p)} \right)
\end{equation*}
has a solution and such a solution does exist if 
$N$ is divisible by $c\prod_{p|3c_1c_2c_3}p^{\gamma'(p)}$. For each prime $p|3c_1c_2c_3$, $\gamma(p)$ and $\gamma'(p)$ are as defined in Section 5.
\end{thm}
\begin{thm} \label{t:one3} 
Let $f$ be the cubic form in (\ref{e:one1}). The number $R(0;P)$ of solutions of $f({\bf x})=0$ in the box $|x_i|\leq P$ satisfies
\begin{equation*}
R(0;P)=\delta_0 P^4+P^4{\mathfrak S}_0J_0+O(P^{{4-1/16+\varepsilon}})
\end{equation*}
where $\delta_0>0$ is as described in the paragraph preceding the statement of theorem 1,
$${\mathfrak S}_0=\sum_{q=1}^{\infty}\sum_{\substack{a=1\\(a,q)=1}}
^{q}
q^{-7}S(q,a) \mbox{ with }S(q,a)=\sum_{{\bf z} \bmod q} e\left( \frac{a}{q}f({\bf z})\right)$$
and
$$J_0=\int_{-\infty}^{\infty}\left( \int_{[-1,1]^7}e(\gamma f({\bf \xi}))\,d{\bf
\xi} \right) d\gamma.$$
\end{thm}
The term $\delta_1 N^{4/3}\chi(N)$ in Theorem 1 and the term $\delta_0 P^4$ in Theorem 3 come from the contribution of the 4-dimensional rational linear spaces contained in the variety $f=0$. It was pointed out to me by the referee of \cite{c7i} that the asymptotic formula in Theorem 3 agrees with Manin's conjecture \cite{Manin} (which applies to the complement of these rational linear spaces in the variety $f=0$).\\
\indent In what follows, we take $P$ to be any large positive number when considering
the representations of zero; we take $P=[N^{1/3}]$ when considering the representations
of a large positive integer $N$. $I$ denotes the set of values $x_i$ are allowed
to take; ${\mathfrak B}=[0,1]^7$ in case $I=\{1,2,\ldots,P\}$
or $\{0,1,2,\ldots,P\}$; ${\mathfrak B}=[-1,1]^7$ in case
$I=\{-P,\ldots,-1,0,1,\ldots,P\}$; $N_1,N_2,\chi(\cdot)$ are as defined in
(\ref{e:one3}),(\ref{e:one4}),(\ref{e:one5}); $\delta=\delta_1$ or $\delta_0$ according as $N$ is a large positive integer or zero. Let $B=I^7$ and
\begin{equation}\label{e:one9}
F(\alpha)= \sum_{{\bf x}\in B} e(\alpha f({\bf x})).
\end{equation}
The number of representations of $N$  by
the cubic form $f$ with $x_i \in I$ is
\begin{equation}\label{e:one10}
R(N)=R(N;I)=\int_{0}^{1} F(\alpha)\,e(-N\alpha)\,d\alpha.
\end{equation}
Note that
$$F(\alpha)=(N_1+F_1(\alpha))(N_2+F_2(\alpha))F_3(\alpha)$$
\begin{equation}\label{e:one11} =N_1N_2F_3(\alpha)+N_1F_2(\alpha)F_3(\alpha)+N_2F_1(\alpha)F_3(\alpha)
+F_1(\alpha)F_2(\alpha)F_3(\alpha)
\end{equation}
where for $i=1,2$
$$
F_i(\alpha)=\sum_{\substack{x,y,z \in I \\ L_i(x,y,z)Q_i(x,y,z)\neq 0}} e(\alpha
L_i(x,y,z)Q_i(x,y,z))=\sum_{|n| \ll P^3}c_i(n)e(\alpha n)$$
with $c_i(0)=0$ and
\begin{equation} \label{e:one12}
c_i(n) = \mbox{Card}\{(x,y,z)\in I^3: L_i(x,y,z)Q_i(x,y,z)=n\} \mbox{ for }
n\neq 0;
\end{equation}
\begin{equation} \label{e:one13}
F_3(\alpha)=\sum_{x \in I} e(\alpha a_7x^3)=\sum_{n \in T} e(\alpha n).
\end{equation}
We need to estimate the contribution of each term on the right in (\ref{e:one11})
to the integral in (\ref{e:one10}). We note here that
\begin{equation} \label{e:one14}
\int_{0}^{1} N_1N_2F_3(\alpha)\,e(-N\alpha)\,d\alpha=N_1N_2 \chi (N)=\delta_1 (P^4+O(P^{3+\varepsilon}))\chi (N),
\end{equation}
\begin{equation} \label{e:one15}
\int_{0}^{1} N_1F_2(\alpha)F_3(\alpha)\,e(-N\alpha)\,d\alpha=N_1\sum_{n \in T} 
c_2(N-n)=N_1\sum_{x_7 \in I} 
c_2(N-a_7x_7^3)
\end{equation}
\begin{equation*}
\int_{0}^{1} N_2F_1(\alpha)F_3(\alpha)\,e(-N\alpha)\,d\alpha=N_2\sum_{n \in T} 
c_1(N-n)=N_2\sum_{x_7 \in I} 
c_1(N-a_7x_7^3)
\end{equation*}
\begin{equation*}
= N_2\sum_{x_7 \in I} \mathrm{Card}\{(x_1,x_2,x_3)\in I^3:L_1Q_1\neq 0,\, L_1Q_1=N-a_7x_7^3\}
\end{equation*}
\begin{equation} \label{e:one16} 
=N_2 \mathrm{Card}\{(x_1,x_2,x_3,x_7) \in I^4:L_1Q_1\neq 0,\, L_1Q_1+a_7x_7^3=N\}
\end{equation}
and, since $N_1,N_2 \ll P^2$, the trivial bounds on $F_1,F_2,F_3$ give
\begin{equation} \label{e:one17}
F(\alpha)-F_1(\alpha)F_2(\alpha)F_3(\alpha)\ll P^6.
\end{equation}
\Section{Auxiliary Transformations}
\label{auxiliary}
\indent As in \cite{c7i}, it will be useful to apply
some linear transformation to the terms $L_iQ_i$ and the equations of the form $L_iQ_i=n_i$, $i=1, 2$ while still maintaining the integrality of the coefficients. We collect here the result of applying these transformations and of various cases that arise and shall refer to them as needed.\\
\indent Consider the equation $L_2Q_2=n_2$. Substitute
\begin{equation} \label{e:two21}
x_4'=L_2=L_2(x_4,x_5,x_6)
\end{equation}
and note that, by renaming the variables if necessary, we can assume that
$a_4\neq0$. Then $x_4=(x_4'-a_5x_5-a_6x_6)/a_4$ and we have
$$a_4^2n_2=a_4^2x_4'Q_2
=x_4'Q_2(a_4x_4,a_4x_5,a_4x_6)=x_4'Q_2(x_4'-a_5x_5-a_6x_6,a_4x_5,a_4x_6).$$
Thus
\begin{equation} \label{e:two22}
a_4^2n_2=a_4^2x_4'Q_2=x_4'(A''x_5^2+B''x_5x_6+C''x_6^2+G''
x_4'x_5+F''
x_4'x_6+A_4x_4'^2)
\end{equation}
where
\begin{equation*}
A''=A_4a_5^2+A_5a_4^2-B_6a_4a_5,\mbox{ }
B''=2A_4a_5a_6+B_4a_4^2-B_5a_4a_5-B_6a_4a_6,
\end{equation*}
\begin{equation*}
C''=A_4a_6^2+A_6a_4^2-B_5a_4a_6, \mbox{ } F''=B_5a_4-2A_4a_6, \mbox{ }
G''=B_6a_4-2A_4a_5.
\end{equation*}
\indent By a tedious but straightforward calculation, 
$B''^2-4A''C''=a_4^2{\bf a}_2{\bf M}_2{\bf a}_2^T=a_4^2\Delta_2$. If $\Delta_2 \neq 0,$ then $L_2Q_2$ is nondegenerate. If $\Delta_2=0$, then $L_2Q_2$ is nondegenerate iff either $A''(2A''F''-B''G'')\neq 0$ or $C''(2C''G''-B''F'')\neq 0$.\\
\indent If $\Delta_2 \neq 0$, $A''\neq 0$, (\ref{e:two22}) becomes
\begin{equation} \label{e:two23}
x_4'(a_4^2\Delta_2 x_5'^2-x_6'^2)=4A'' a_4^4\Delta_2 n_2- D''x_4'^3=4A''
a_4^4\Delta_2 x_2'Q_2- D''x_4'^3
\end{equation}
where $$x_5' = 2A''x_5+B''x_6+G''x_4',\mbox{ }
x_6'= a_4^2\Delta_2 x_6+(B''G''-2A''F'')x_4',$$ and
$$D''=a_4^2\Delta_2(4A''A_4-G''^2)+(2A''F''-B''G'')^2.$$
Also, in this case
$$4A'' a_4^4\Delta_2(L_2Q_2+a_7x_7^3)=L_2(a_4^2\Delta_2 x_5'^2-x_6'^2)+D''L_2^3+4A'' a_4^4\Delta_2a_7x_7^3$$
so that in case $\Delta_2$ is a square and $D''/4A'' a_4^4\Delta_2a_7$ is a nonzero rational cube, say $\Delta_2=d^2$, $D''=E''d_1^3$, $4A'' a_4^4\Delta_2a_7=E''d_2^3$, $E''\neq 0$, we have 
\begin{equation} \label{e:two24}
4A'' a_4^4\Delta_2(L_2Q_2+a_7x_7^3)=L_2x_5''x_6''+E''(d_1^3L_2^3+d_2^3x_7^3)
\end{equation}
where $x_5''=a_4dx_5'+x_6'$ and $x_6''=a_4dx_5'-x_6'$.\\
\indent If $\Delta_2 \neq 0$, $A''=0$ 
but $C'' \neq 0$, we can again complete the square to get
\begin{equation} \label{e:two25}
x_4'(a_4^2\Delta_2x_6'^2- x_5'^2)=4a_4^4\Delta_2C''n_2-D''x_1''^3=4a_4^4\Delta_2C''x_2'Q_2- D''x_4'^3
\end{equation}
where $$x_5'= a_4^2\Delta_2 x_5+(B''F''-2C''G'')x_4', \mbox{ }x_6' = 2C''x_6+B''x_5+F''x_4',$$ and
$$D''=a_4^2\Delta_4(4C''A_4-F''^2)+(2C''G''-B''F'')^2.$$
Also, in this case
$$4C'' a_4^4\Delta_2(L_2Q_2+a_7x_7^3)=L_2(a_4^2\Delta_2 x_6'^2-x_5'^2)+D''L_2^3+4C'' a_4^4\Delta_2a_7x_7^3$$
so that in case $\Delta_2$ is a square and $D''/4C'' a_4^4\Delta_2a_7$ is a nonzero rational cube, say $\Delta_2=d^2$, $D''=E''d_1^3$, $4C'' a_4^4\Delta_2a_7=E''d_2^3$, $E''\neq 0$, we have 
\begin{equation} \label{e:two26}
4C'' a_4^4\Delta_2(L_2Q_2+a_7x_7^3)=L_2x_5''x_6''+E''(d_1^3L_2^3+d_2^3x_7^3)
\end{equation}
where $x_5''=a_4dx_6'+x_5'$ and $x_6''=a_4dx_6'-x_5'$.\\
\indent If $A''=C''=0$, $B''\neq 0$, then multiplying (\ref{e:two22}) by $B''$ and adding a suitable multiple of $x_4'^3$ to both sides we get,
\begin{equation} \label{e:two27}
x_4'x_5'x_6'=B''a_4^2n_2 -D''x_4'^3=B''a_4^2x_4'Q_2 -D''x_4'^3
\end{equation}
where $$x_5'=B''x_5+F''x_4',\mbox{ } x_6'=B''x_6+G''x_4',\mbox{ } D''=B''A_4-F''G''.$$
Also, in this case,
$$B''a_4^2(L_2Q_2+a_7x_7^3)=L_2x_5'x_6'+D''L_2^3+B''a_4^2a_7x_7^3,$$
$a_4^2\Delta_2=B''^2$ is a square and in case $D''/B''a_4^2a_7$ is a nonzero rational cube, say $D''=E''d_1^3$, $B''a_4^2a_7=E''d_2^3$, $E''\neq 0$, writing $x_5''$ for $x_5'$ and $x_6''$ for $x_6'$, we have 
\begin{equation} \label{e:two28}
B''a_4^2(L_2Q_2+a_7x_7^3)=L_2x_5''x_6''+E''(d_1^3L_2^3+d_2^3x_7^3).
\end{equation}
\indent If $\Delta_2=0$, $A''(2A''F''-B''G'')\neq 0$, then
\begin{equation*}
4A''a_4^2Q_2(x_4,x_5,x_6)=x_4'^2(4A''A_4-G''^2)+(4A''F''-2B''G'')x_4'x_6+x_5'^2=x_4'x_6'+x_5'^2,
\end{equation*}
\begin{equation} \label{e:two29}
4A''a_4^2x_4'Q_2(x_4,x_5,x_6)=4A''a_4^2n_2
=x_4'(x_4'x_6'+x_5'^2)
\end{equation}
where $$x_5'=2A''x_5+B''x_6+G''x_4', \mbox{ }x_6'=(4A''A_4-G''^2)x_4'+(4A''F''-2B''G'')x_6.$$
\indent If $\Delta_2=0$, $C''(2C''G''-B''F'')\neq 0$, then
\begin{equation*}
4C''a_4^2Q_2(x_4,x_5,x_6)=x_4'^2(4C''A_4-F''^2)+(4C''G''-2B''F'')x_4'x_6+x_5'^2=x_4'x_6'+x_5'^2,
\end{equation*}
\begin{equation} \label{e:two210}
4C''a_4^2x_2'Q_2(x_4,x_5,x_6)=4C''a_4^2n_2
=x_4'(x_4'x_6'+x_5'^2)
\end{equation}
where $$x_5'=2C''x_6+B''x_5+F''x_4', \mbox{ } x_6'=(4C''A_4-F''^2)x_4'+(4C''G''-2B''F'')x_5.$$
\indent Now consider the equation $L_1Q_1=n_1$. Substitute
\begin{equation} \label{e:two1}
x_1'=L_1=L_1(x_1,x_2,x_3)=a_1x_1+a_2x_2+a_3x_3.
\end{equation}
By renaming the variables if necessary, we can assume that
$a_1\neq0$. Then 
$$a_1^2n_1=a_1^2x_1'Q_1
=x_1'Q_1(a_1x_1,a_1x_2,a_1x_3)=x_1'Q_1(x_1'-a_2x_2-a_3x_3,a_1x_2,a_1x_3)$$
\begin{equation} \label{e:two2}
=x_1'(A'x_2^2+B'x_2x_3+C'x_3^2+G'
x_1'x_2+F'
x_1'x_3+A_1x_1'^2)
\end{equation}
where
\begin{equation*}
A'=A_1a_2^2+A_2a_1^2-B_3a_1a_2,\mbox{ }
B'=2A_1a_2a_3+B_1a_1^2-B_2a_1a_2-B_3a_1a_3,
\end{equation*}
\begin{equation}  \label{e:two10}
C'=A_1a_3^2+A_3a_1^2-B_2a_1a_3, \mbox{ } F'=B_2a_1-2A_1a_3, \mbox{ }
G'=B_3a_1-2A_1a_2.
\end{equation}
\indent We are assuming that $B'^2-4A'C'=a_1^2{\bf a}_1{\bf M}_1{\bf a}_1^T=a_1^2\Delta_1=0$. Note that if $A'=C'=0$ then $B'=0$ as well and $L_1Q_1$ is degenerate. In fact, in case $B'^2-4A'C'=0$, $L_1Q_1$ is nondegenerate iff either $A'(2A'F'-B'G')\neq 0$ or $C'(2C'G'-B'F')\neq 0.$\\
\indent If $\Delta_1=0$, $A'(2A'F'-B'G')\neq 0$, then
\begin{equation*}
4A'a_1^2Q_1(x_1,x_2,x_3)=x_1'^2(4A'A_1-G'^2)+(4A'F'-2B'G')x_1'x_3+x_2'^2=x_1'x_3'+x_2'^2,
\end{equation*}
\begin{equation} \label{e:two7}
4A'a_1^2x_1'Q_1(x_1,x_2,x_3)=4A'a_1^2n_1
=x_1'(x_1'x_3'+x_2'^2)
\end{equation}
where $$x_2'=2A'x_2+B'x_3+G'x_1', \mbox{ }x_3'=(4A'A_1-G'^2)x_1'+(4A'F'-2B'G')x_3.$$
\indent If $\Delta_1=0$, $C'(2C'G'-B'F')\neq 0$, then
\begin{equation*}
4C'a_1^2Q_1(x_1,x_2,x_3)=x_1'^2(4C'A_1-F'^2)+(4C'G'-2B'F')x_1'x_3+x_2'^2=x_1'x_3'+x_2'^2,
\end{equation*}
\begin{equation} \label{e:two8}
4C'a_1^2x_1'Q_1(x_1,x_2,x_3)=4C'a_1^2n_1
=x_1'(x_1'x_3'+x_2'^2)
\end{equation}
where $$x_2'=2C'x_3+B'x_2+F'x_1', \mbox{ } x_3'=(4C'A_1-F'^2)x_1'+(4C'G'-2B'F')x_2.$$
\Section{The Treatment of the Minor Arcs}
\label{minorarcs}
\noindent Let $\delta$ be a sufficiently small positive real number to be specified 
later. For $1 \leq q \leq P^{\delta},\; 1 \leq a \leq q,\; (a,q)=1,$ define the major arcs $\mathfrak{M}(q,a)$ by
\begin{equation} \label{e:three1}
\mathfrak{M}(q,a)=\left\{\alpha:\left| \alpha
- \frac{a}{q}\right| \leq P^{\delta-3}\right\}
\end{equation}
and let $\mathfrak{M}$ be the union of the $\mathfrak{M}(q,a)$. The set $\mathfrak{m}=(P^{\delta - 3},
1+P^{\delta - 3}]\backslash \mathfrak{M}$ forms the minor arcs. The following two lemmas give the estimate for the minor arc contribution from the last term in (\ref{e:one11}).
\begin{lem} \label{l:three1} For $i=1,2$ we have
\begin{equation} \label{e:three2}
\int_{0}^{1} \left| F_i(\alpha)\right|^2 d\alpha \ll P^{3+\varepsilon}.
\end{equation}
\end{lem}
\begin{proof} We considered the case $\Delta_i \neq 0$ in \cite{c7i}; we consider the case $\Delta_i=0$ here. We have 
\begin{equation} \label{e:three3}
\int_{0}^{1}|F_1(\alpha)|^2 d\alpha
=\sum_{|n| \ll P^3}c_1(n)^2
=\sum_{|n| \ll P^3,\, n\neq 0}c_1(n)^2
\end{equation}
with $c_1(n)$ as defined in (\ref{e:one12}), $c_1(0)=0$. For $n\neq 0, x\neq 0$, let $$R(n,x):=\mbox{Card} \{(y,z): |y|,|z| \leq P, x(xy+z^2)=n\}.$$
Then, by (\ref{e:two7}), (\ref{e:two8}) and Cauchy's inequality
$$\sum_{\substack{|n|\ll P^3\\ n \neq 0}}c_1(n)^2 \leq \sum_{\substack{|n|\ll P^3\\ n \neq 0}} \left(\sum_{\substack{|x|\leq P\\ x|n}} R(n,x)\right)^2 \leq \sum_{\substack{|n|\ll P^3\\ n\neq 0}} \sum_{x|n}1 \sum_{\substack{|x|\leq P\\ x|n}} R(n,x)^2$$
$$\ll P^{\varepsilon}\sum_{\substack{|n|\ll P^3\\ n\neq 0}} \sum_{\substack{|x|\leq P\\ x|n}} R(n,x)^2
\leq 
P^{\varepsilon}\sum_{\substack{|x|\leq P\\ x\neq 0}} \sum_{m\neq 0} R(mx,x)^2$$
$$=P^{\varepsilon}\mbox{Card} \{(x,y,z,y',z'): |x|,|y|,|z|,|y'|,|z'| \leq P,\, x\neq 0,\, x(xy+z^2)=x(xy'+z'^2)\}$$
$$=P^{\varepsilon}\mbox{Card} \{(x,y,z,y',z'): |x|,|y|,|z|,|y'|,|z'| \leq P,\, x\neq 0,\, x(y-y')=z'^2-z^2\}.$$
In the last expression, for fixed $(z,z')$ the number of choices for $(x,y,y')$ is $O(P^2)$ if $z=z'$ and $O(P^{1+\varepsilon})$ if $z\neq z'.$ Hence, the inequality (\ref{e:three2}) holds.
\qedhere
\end{proof}
\begin{lem} \label{l:three2}
We have
$$\int_{\mathfrak{m}}F_1(\alpha)F_2(\alpha)F_3(\alpha)\,e(-N\alpha)\,d\alpha \ll P^{4+\varepsilon-(\delta/4)}.$$
\end{lem}
\begin{proof}
For $\alpha \in \mathfrak{m}$,
by Dirichlet's theorem on Diophantine Approximation, there exist $a,q$ with
$(a,q) =1, q\leq P^{3-\delta}$ and $\left| \alpha
- \frac{a}{q}\right| \leq q^{-1}P^{\delta-3} < q^{-2}.$ Since
$\alpha \in \mathfrak{m}\, \subseteq (P^{\delta - 3},
1-P^{\delta - 3})$, we must have $1\leq a < q$ and $q
>P^{\delta}$ (otherwise $\alpha$ would be in $\mathfrak{M}(q,a)$). By Weyl's
inequality (See \cite{VaughanBook}),
$$F_3(\alpha)\ll P^{1+\varepsilon-(\delta/4)}.$$
This, with lemma \ref{l:three1} and Cauchy's inequality, gives
$$\int_{\mathfrak{m}}F_1(\alpha)F_2(\alpha)F_3(\alpha)\,e(-N\alpha)\,d\alpha
\leq \left( \sup_{\alpha\in\mathfrak{m}}|F_3(\alpha)|
\right) \left( \int_{\mathfrak{m}} |F_1(\alpha)||F_2(\alpha)| d\alpha \right)$$
$$\leq \left( \sup_{\alpha\in\mathfrak{m}}|F_3(\alpha)| \right)
\left( \int_{0}^{1} |F_1(\alpha)|^2 d\alpha \right)^{\frac{1}{2}}
\left( \int_{0}^{1} |F_2(\alpha)|^2 d\alpha \right)^{\frac{1}{2}}
\ll P^{4+\varepsilon-\frac{\delta}{4}}. \qedhere$$
\end{proof}
\indent We need the following two lemmas to estimate the contribution from the second and third terms in (\ref{e:one11}).
\begin{lem} \label{l:three0}
Let $k\geq 2, q\in \mathbb{N}$. For any $m\in \mathbb{Z}$, the number of solutions of the congruence
\begin{equation} \label{e:three0}
x^k \equiv m \pmod q
\end{equation}
is $\ll q^{1-\frac{1}{k}}.$
\end{lem}
\begin{proof}
It is well known that the congruence (\ref{e:three0}) has $\ll q^{\varepsilon}$ solutions if $(m,q)=1$. If $(m,q)=D$, write $D=d_1d_2^2\ldots d_k^k$ with $d_1, d_2,\ldots , d_{k-1}$ square-free and relatively prime in pairs. Then $d_1d_2\ldots d_k|x$ for any solutions $x$ of (\ref{e:three0}). Let
$$x_0=x/d_1d_2\ldots d_k, \mbox{ } q_0=q/D, \mbox{ } m_0=m/D.$$
Then $(m_0,q_0)=1$ and (\ref{e:three0}) becomes
$$d_1^{k-1}d_2^{k-2}\ldots d_{k-1} x_0^k \equiv m_0 \pmod{q_0}$$
which has $\ll q_0^{\varepsilon}=(q/D)^{\varepsilon}$ solutions with $1\leq x_0\leq q_0=q/D.$ Thus the number of solutions of (\ref{e:three0}) with $1\leq x\leq q$ is $\ll (q/D)^{\varepsilon}D/d_1d_2\ldots d_k=(q/D)^{\varepsilon}d_2d_3^2\ldots d_k^{k-1}\leq (q/D)^{\varepsilon}D^{1-\frac{1}{k}}=q^{\varepsilon}D^{1-\frac{1}{k}-\varepsilon}\leq q^{\varepsilon}q^{1-\frac{1}{k}-\varepsilon}=q^{1-\frac{1}{k}}.$
\end{proof}
\begin{lem} \label{l:three3}
For any fixed non-zero integer $A$,
$$\mathrm{Card}\{(x,y,z,w):|x|,|y|,|z|,|w|\leq P,\,x(xy+z^2)\neq 0,\, x(xy+z^2)+Aw^3=N\}\ll P^{\frac{11}{6}+\varepsilon}.$$
\end{lem}
\begin{proof} Let $\theta$ be a parameter, $0<\theta<1$, to be chosen later. We estimate the number of those solutions with $0<|x|\leq P^{\theta}$ and those with $P^{\theta}<|x|\leq P$ separately.\\
\indent By Theorem 1 of \cite{VaughanEll}, for a fixed pair $(x,y)$ the equation $x(xy+z^2)+Aw^3=N$ (which is equivalent to $xz^2=-Aw^3+N-x^2y$) has $\ll P^{\frac{1}{2}}$ solutions in $(z,w)$ with $|z|\leq P$. Thus there are $\ll P^{\frac{3}{2}+\theta}$ solutions of $x(xy+z^2)+Aw^3=N$ with $0<|x|\leq P^{\theta}$.\\
\indent In the equation $x(xy+z^2)+Aw^3=N$, $y$ gets fixed once $x, z, w$ are chosen. There are at most $2P+1$ choices for $w$. Having chosen $w$, there are no solutions in $x, y, z$ with $|x|,|y|,|z|\leq P,\,x(xy+z^2)\neq 0$ unless $0<|N-Aw^3|\leq 2P^3$ in which case $x$ must still be a (positive or negative) divisor of $N-Aw^3$ giving $\ll P^{\varepsilon}$ choices for $x$ (with or without the condition $P^{\theta}<|x|\leq P$ though we are interested only in those $x$ which satisfy this additional condition) and then $z$ must be a solution of the congruence
\begin{equation*}
z^2\equiv \frac{N-Aw^3}{x} \pmod x
\end{equation*}
with $|z|\leq P$ giving $\ll \frac{P}{\sqrt{x}}+\sqrt{x}\ll \frac{P}{\sqrt{x}} \ll P^{1-\frac{\theta}{2}}$ choices for $z$ (by the case $k=2$ of lemma \ref{l:three0}). Thus there are $\ll P^{2-\frac{\theta}{2}+\varepsilon}$ solutions of $x(xy+z^2)+Aw^3=N$ with $P^{\theta}<|x|\leq P$.\\
\indent The lemma follows on adding the two estimates and taking $\theta =\frac{1}{3}$.
\qedhere
\end{proof}
\indent By (\ref{e:one16}), (\ref{e:two7}), (\ref{e:two8}) and lemma \ref{l:three3} we have
\begin{equation} \label{e:three5}
\int_{0}^{1} N_2F_1(\alpha)F_3(\alpha)\,e(-N\alpha)\,d\alpha \ll N_2P^{\frac{11}{6}+\varepsilon}
\ll P^{\frac{23}{6}+\varepsilon}.
\end{equation}
From the argument in \cite{c7i}, if $\Delta_2 \neq 0$ and $N$ is a large positive integer we have
\begin{equation} \label{e:three6}
\int_{0}^{1} N_1F_2(\alpha)F_3(\alpha)\,e(-N\alpha)\,d\alpha \ll N_1P^{1+3\varepsilon}
\ll P^{3+3\varepsilon}
\end{equation}
while in case $\Delta_2 \neq 0$ and $N=0$ we have
\begin{equation} \label{e:three00}
\int_{0}^{1} N_1F_2(\alpha)F_3(\alpha)\,e(-N\alpha)\,d\alpha = \delta_2 P^4+O(P^{3+\varepsilon}).
\end{equation}
In case $\Delta_2=0$, we have
\begin{equation} \label{e:three18}
\int_{0}^{1} N_1F_2(\alpha)F_3(\alpha)\,e(-N\alpha)\,d\alpha \ll N_1P^{\frac{11}{6}+\varepsilon}
\ll P^{\frac{23}{6}+\varepsilon}.
\end{equation}
\indent Since the total length of the major arcs is less that $P^{3\delta -3}$, from
(\ref{e:one17}),
$$\int_{0}^{1}
F_1(\alpha)F_2(\alpha)F_3(\alpha)\,e(-N\alpha)\,d\alpha=
\left(\int_{\mathfrak m}+\int_{\mathfrak M}\right)
F_1(\alpha)F_2(\alpha)F_3(\alpha)\,e(-N\alpha)\,d\alpha$$
$$=\int_{\mathfrak m}
F_1(\alpha)F_2(\alpha)F_3(\alpha)\,e(-N\alpha)\,d\alpha+\int_{\mathfrak
M}F(\alpha)\,e(-N\alpha)\,d\alpha+O(P^{3+3\delta}).$$
\noindent Along with (\ref{e:one11}), (\ref{e:one14}), (\ref{e:three5}),
(\ref{e:three6}), (\ref{e:three00}), (\ref{e:three18}) and lemma \ref{l:three2}, this gives
\begin{equation} \label{e:three7}
R(N)=\delta P^4\chi(N)+\int_{\mathfrak
M}F(\alpha)\,e(-N\alpha)\,d\alpha+O(P^{\frac{23}{6}+\varepsilon}+P^{3+3\delta}+P^{4+\varepsilon-\frac{\delta}{4}})
\end{equation}
where $\delta=\delta_1$ in case $N$ is a positive integer, $\delta=\delta_1+\delta_2=\delta_0$ in case $N=0$.
\Section{The Singular Series}
\label{sseries}
\noindent Define ${\mathfrak{S}}(N,Q)$ and $S(q;N)$ by
\begin{equation} \label{e:six1}
{\mathfrak{S}}(N,Q)=\sum_{q\leq Q}S(q;N)=\sum_{q\leq Q}\sum_{\substack{a=1\\(a,q)=1}}^{q} q^{-7}S(q,a)e(-aN/q)
\end{equation}
where 
\begin{equation} \label{e:six2}
S(q,a)=\sum_{{\bf z} \bmod q} e\left( \frac{a}{q}f({\bf z})\right) 
.
\end{equation}
We have the following factorization.
\begin{equation} \label{e:six3}
S(q,a)= S_1(q,a)S_2(q,a)S_3(q,a)
\end{equation}
where $$S_i(q,a)=\sum_{x,y,z = 1}^{q} e \left(\frac{a}{q} L_i(x,y,z)Q_i(x,y,z)\right)$$
for $i=1,2$ and $$S_3(q,a)=\sum_{x = 1}^{q} e\left( \frac{aa_7}{q}x^3 \right).$$
\noindent $S(q;N)$ is a multiplicative function of $q$.
In fact, if $(q_1,q_2)=(a,q_1q_2)=1$ and we choose $a_1,a_2$ so that $a_2q_1+a_1q_2=1$
then for $i=1,2,3
$\begin{equation} \label{e:six4}
S_i(q_1q_2,a)=S_i(q_1q_2,a_1q_2+a_2q_1)=S_i(q_1,a_1)S_i(q_2,a_2).
\end{equation}
This follows from the fact that as $r_1$ and $r_2$ run through
a complete (or reduced) set of residues modulo $q_1$ and $q_2$ respectively, $r_2q_1+r_1q_2$ runs through
a complete (or reduced) set of residues modulo $q_1q_2$.
Thus it suffices to study $S_i(q,a)$ when $q$ is a prime power. We can also
assume that $L_1Q_1$, $L_2Q_2$ have content 1. For, if $g({\bf x})
$ is a polynomial in $s$ variables with content $c$, say
$g({\bf x})=cg_1({\bf x})$, $(a,q)=1$, $(c,q)=c_0$ then $c_0$ is bounded in 
terms of $g$ alone and writing $c=c_0c'$ we have $(q/c_0,ac')=1$ and
$$\sum_{{\bf x} \bmod q} e\left( \frac{ag({\bf x})}{q} \right)
=\sum_{{\bf x} \bmod q} e\left( \frac{acg_1({\bf x})}{q} \right)
=\sum_{{\bf x} \bmod q} e\left( \frac{ac'g_1({\bf x})}{q/c_0} \right)$$
$$=c_0^s\sum_{{\bf x} \bmod (q/c_0)} e\left( \frac{ac'g_1({\bf x})}{q/c_0}\right).$$
Thus a bound in content 1 case leads to a bound in the general case with
a different implied constant but which depends on the polynomial $g$ alone. For example,
if $p \nmid a$, $(a,q)=1$, the standard bounds for $\sum_{x = 1}^{q} e(ax^3/q)$ (see \cite{VaughanBook}) lead to
\begin{equation} \label{e:six0}
S_3(p,a)\ll p^{1/2} \mbox{ , } S_3(p,a)\ll p
\end{equation}
and, in general
\begin{equation} \label{e:six5}
|S_3(q,a)|\leq C (a_7,q)(q/(a_7,q))^{2/3}=C(a_7,q)^{1/3}q^{2/3} \leq
Ca_7^{1/3}q^{2/3} \ll q^{2/3}.
\end{equation}
\indent Thus, in the sum
$S_1(p^k,a)
$,
by renaming the variables if needed, we can further assume that $p\nmid a_1.$
Let
\begin{equation} \label{e:six6}
{\mathfrak D_1}=\left\{ \begin{array}{lcl}
2\Delta_1A' & \mbox{ if } & \Delta_1A'\neq 0\\
2\Delta_1C' & \mbox{ if } & A'=0 \mbox{ but } \Delta_1C'\neq 0\\
B' & \mbox{ if } & A'=C'=0,\, B' \neq 0 \mbox{ (so } \Delta_1 \neq 0)\\
2A' & \mbox{ if } & \Delta_1=0,\, A'(2A'F'-B'G')\neq 0\\
2C' & \mbox{ if } & \Delta_1=0,\, C'(2C'G'-B'F') \neq 0\\
\end{array}
\right.
\end{equation}
and consider first the primes $p$ not dividing ${\mathfrak D_1}.$ In \cite{c7i} we showed that
\begin{equation} \label{e:six7}
 |S_1(q,a)|<q^2d(q)\ll q^{2+\varepsilon} \mbox{ if }\Delta_1\neq 0 \mbox{ and } (q,{\mathfrak D_1})=1.
\end{equation}
Now we consider the case $\Delta_1=0$, $A'(2A'F'-B'G')\neq 0$; the case $\Delta_1=0$, $C'(2C'G'-B'F')\neq 0$ is similar. With $x_1',x_2',x_3'$ as in
(\ref{e:two1}), (\ref{e:two7}), we have
$$4A'a_1^2L_1(x_1,x_2,x_3)Q_1(x_1,x_2,x_3)=x_1'(x_1'x_3'+x_2'^2)$$
and the determinant of the linear transformation
$(x_1,x_2,x_3)\mapsto (x_1',x_2',x_3')$ is $2a_1A'(2A'F'-B'G')$. Since $p \nmid {\mathfrak D_1}= 2A'(2A'F'-B'G')$,
as $(x_1,x_2,x_3)$ runs through a complete residue system modulo $p^k$ so does
$(x_1',x_2',x_3')$ and as $(u,v,w)$ runs through a complete residue system so
does $(4a_1^2A'u,v,w)$. Hence,
$$S_1(p^k,a)=\sum_{x_1,x_2,x_3=1}^{p^k}e\left(
\frac{a}{p^k}L_1(x_1,x_2,x_3)Q_1(x_1,x_2,x_3)\right)$$
$$=\sum_{x_1',x_2',x_3'=1}^{p^k}e\left(
\frac{ax_1'(x_1'x_3'+x_2'^2)}{p^k4a_1^2A'} \right)=\sum_{u,v,w=1}^{p^k}e\left(
\frac{a4a_1^2A'u(4a_1^2A'uw+v^2)}{p^k4a_1^2A'}\right)$$
$$=\sum_{u,v,w=1}^{p^k}e\left(
\frac{a}{p^k}(A''uv^2+u^2w)\right)=\sum_{u=1}^{p^k}\left\{ \sum_{v=1}^{p^k}e\left(
\frac{aA''uv^2}{p^k}\right) \sum_{w=1}^{p^k}e\left(
\frac{au^2w}{p^k}\right) \right\}
$$
where $A''=4a_1^2A'$. For a fixed $u= p^ju_1$, $p\nmid u_1$, the sum over $v$ is equal to
$$p^j\sum_{v_2=1}^{p^{k-j}}
e\left(\frac{aA''u_1}{p^{k-j}}v_2^2\right)=
\left\{
\begin {array}
{lc}p^{(k+j)/2} & \mbox{ if  $k-j$  is even }\\
p^{(k+j)/2} \left( \frac{aA''u_1}{p} \right) _L & \mbox{ if $k-j$  is odd and
$p\equiv 1 \pmod 4$}\\
\sqrt{-1} \mbox{ } p^{(k+j)/2}\left( \frac{aA''u_1}{p} \right) _L & \mbox{ if $k-j$  is
odd and $p\equiv 3 \pmod
4.$}
\end{array}\right.$$
The sum over $w$ is zero unless $p^k|u^2$ (i.e., $j\geq \lceil \frac{k}{2}\rceil$) in which case it is $p^k.$ Also, for a fixed value of $j$, there is only one value of $u$ mod $p^k$ if $j=k$ while 
there are $p^{k-j}(1-p^{-1})$ values of $u$ mod $p^k$ if $0\leq j\leq k-1$. Hence,
\begin{equation*}
S_1(p^k,a)=p^{2k}+\sum_{\substack{\lceil \frac{k}{2}\rceil \leq j<k \\ k-j \mbox{\small { even}}}}p^kp^{(k+j)/2} p^{k-j}(1-p^{-1})=p^{2k+\lfloor k/4\rfloor}.
\end{equation*}
Hence,
\begin{equation} \label{e:six8}
S_1(p,a)=p^2 \mbox{ if } p\nmid \mathfrak{D_1}
\end{equation}
and
\begin{equation} \label{e:six9}
|S_1(q,a)| \leq q^{\frac{9}{4}} \mbox{ if } (q, \mathfrak{D_1})=1.
\end{equation}
For the (finitely many) primes $p$ dividing ${\mathfrak D_1}$, as shown in \cite{c7i}
\begin{equation} \label{e:six10}
|S_1(p^k,a)|\ll p^{5k/2}
\end{equation}
with the implicit constant depending on the coefficients of the cubic form only.
\begin{lem} \label{l:six1}
The series \begin{equation} \label{e:six12}
{\mathfrak{S}}(N)=\sum_{q=1}^{\infty}S(q;N)
\end{equation} converges
absolutely and uniformly in $N$ (so that ${\mathfrak{S}}(N)\ll 1$), and for any $Q\geq 1$ we have
\begin{equation} \label{e:six13}
{\mathfrak{S}}(N)= {\mathfrak{S}}(N,Q)+O(Q^{-1/3+\varepsilon}).
\end{equation}
\end{lem}
\begin{proof} Writing $q=q_1q_2$ where $$q_1=\prod_{p^k \|q, p\nmid \mathfrak D_1
\mathfrak D_2}p^k,\mbox{	} q_2=\prod_{p^k \|q, p|\mathfrak D_1 \mathfrak D_2}p^k,$$ from (\ref{e:six3})-(\ref{e:six5}), (\ref{e:six7})-(\ref{e:six10}), we have
$$\sum_{q>Q}|S(q;N)|\ll \sum_{q_2} q_2^{-1/3} \sum_{q_1>Q/q_2} S(q_1;N)$$
$$\leq \sum_{q_2} q_2^{-\frac{1}{3}} \left(\frac{Q}{q_2}\right)^{-\frac{1}{3}+\varepsilon}
\sum_{q_1>Q/q_2} S(q_1;N)q_1^{\frac{1}{3}-\varepsilon}
= Q^{-\frac{1}{3}+\varepsilon}\sum_{q_2} q_2^{-\varepsilon} \sum_{q_1} S(q_1;N)q_1^{\frac{1}{3}-\varepsilon}$$
$$=Q^{-\frac{1}{3}+\varepsilon}\prod_{p|\mathfrak D_1 \mathfrak D_2} \left(1-p^{-\varepsilon}\right)^{-1} \prod_{p\nmid \mathfrak D_1 \mathfrak D_2} \left(1+\sum_{k=1}^{\infty}S(p^k;N)p^{(\frac{1}{3}-\varepsilon)k}\right)$$
$$\ll Q^{-\frac{1}{3}+\varepsilon}
\prod_{p\nmid \mathfrak D_1 \mathfrak D_2} \left(1+O\left(p^{2+2+\frac{1}{2}+1-7+\frac{1}{3}-\varepsilon}+p^{4+4+1+2-14+\frac{2}{3}-2\varepsilon}+\sum_{k=3}^{\infty}p^{(\frac{9}{4}+\frac{9}{4}+\frac{2}{3}+1-7+\frac{1}{3}-\varepsilon)k}\right)\right)$$
$$=Q^{-\frac{1}{3}+\varepsilon}
\prod_{p\nmid \mathfrak D_1 \mathfrak D_2} \left(1
+O(p^{-\frac{7}{6}-\varepsilon})\right) \ll Q^{-\frac{1}{3}+\varepsilon}.
\qedhere
$$
\end{proof}
\Section{Completion of the proofs of the theorems} \label{completion}
\noindent The treatment of the major arcs is exactly the same as in \cite{c7i} so we omit the details. The congruence conditions are still given by Theorem 2 of \cite{c7i}; we repeat the (rather long!) definitions of $\gamma$ and $\gamma'$ for completeness. Write $f(\bf x)$ in the form 
$f({\bf x})=c(c_1L_1'Q_1'+c_2L_2'Q_2'+c_3x_7^3)$
where $c$ is the content of $f$ while $L_1'Q_1'$, $L_2'Q_2'$ have content $1$ and let $j_1,j_2,j_3\geq 0$ be integers such that $p^{j_i}\|c_i$ for $i=1,2,3$ and let $\nu_0=\max\{j_1,j_2,j_3\}$. Let $A',B',C',F',G'$ be as in (\ref{e:two10}) but with $L_1Q_1$ replaced by $L_1'Q_1'$. For a prime $p$ dividing $3c_1c_2c_3$, define $\gamma_1$ and $\gamma_1'$ as follows.\\
(i) If $L_1^2|Q_1 \pmod {p}$, choose $\alpha_1, \beta_1$ so that
$p^{\alpha_1}\|(A',B',C')$, $\beta_1=0$ if $(F',G')=(0,0)$, $p^{\beta_1}\|(F',G')$
if $(F',G')\neq (0,0)$. Then
\begin{equation*}
\gamma_1'=\left\{ \begin{array}{lcl}
\lceil \frac{5\alpha_1+1}{3}\rceil & \mbox{ if } & \beta_1=0\\
\max \{\lceil \frac{5\alpha_1+1}{3} \rceil, \lceil
\frac{4\alpha_1+1-\beta_1}{2} \rceil \}
 & \mbox{ if } & \beta_1 \geq 1
\end{array}
\right.,
\gamma_1=\left\{ \begin{array}{lcl}
2 \gamma_1'+1 & \mbox{ if } & p=3\\
2 \gamma_1'-1
 & \mbox{ if } & p \neq 3.
\end{array}
\right.
\end{equation*}
(ii) If $p=2$ and $L_1'Q_1'$ is
$$L_1'x_2(L_1'+x_2), \mbox{ } L_1'x_3(L_1'+x_3), \mbox{ }
L_1'(x_2^2+x_3^3+L_1'x_2+L_1'x_3)=L_1'x_2(L_1'+x_2)+L_1'x_3(L_1'+x_3)$$
when reduced mod 2 (for example, $L_1Q_1=x_1(x_1x_2+(x_2+2x_3)^2)$), then $\gamma_1=\gamma_1'=1.$\\
(iii) If $p=3$ and when reduced mod 3 $L_1'Q_1'$ is equivalent to
$$L_1'(L_1'^2+2x_2'^2) \mbox{ or } 2L_1'(L_1'^2+2x_2'^2)$$
for some linear combination $x_2'$ of $x_1,x_2,x_3$ (for example, $L_1Q_1=x_1(2x_1(x_1+3x_2)+x_3^2$)), then $\gamma_1=3,\gamma_1'=1.$\\
(iv) In other cases, let $\gamma_1=\gamma_1'=0.$\\
\noindent Define $\gamma_2$ and $\gamma_2'$ similarly in terms of the coefficients of $L_2'Q_2'$. Finally, let
$$\gamma'=\gamma'(p,f)=\min\{\gamma_1'+j_1,\gamma_2'+j_2\}$$
and
$$\gamma=\gamma(p,f)=
\left\{
\begin{array}{cl} \min\{\gamma_1+\nu_0,\gamma_2+\nu_0, 2\gamma'+1\} & \mbox{if } p=3\\
\min\{\gamma_1+\nu_0,\gamma_2+\nu_0, 2\gamma'-1\} & \mbox{otherwise.}
\end{array}\right.$$
\begin{small}
{\bf Acknowledgments.} I would like to thank Prof. Robert C. Vaughan for providing me the manuscript of his yet to be published work \cite{VaughanEll} and for many useful discussions while I was his student.
\end{small}

\end{document}